\documentclass[12pt,reqno]{amsart}

\usepackage{amsfonts, amsthm, amsmath}
\allowdisplaybreaks[4]

\usepackage{rotating}

\usepackage{tikz}

\usepackage{graphics}

\usepackage{amssymb}

\usepackage{amscd}

\usepackage[latin2]{inputenc}

\usepackage{t1enc}

\usepackage[mathscr]{eucal}

\usepackage{ferrers}

\usepackage{indentfirst}

\usepackage{graphicx}

\usepackage{graphics}

\usepackage{pict2e}

\usepackage{mathrsfs}

\usepackage{enumerate}
\usepackage[pagebackref]{hyperref}
\hypersetup{backref, pagebackref, colorlinks=true}
\usepackage{cite}
\usepackage{color}
\usepackage{epic}
\usepackage{hyperref} %this gives clickable references, which is nice
\numberwithin{equation}{section}
\theoremstyle{plain}

\newtheorem{theorem}{Theorem}[section]

\newtheorem{lemma}[theorem]{Lemma}

\newtheorem{corollary}[theorem]{Corollary}

\theoremstyle{definition}

\newtheorem{Def}[theorem]{Definition}

\newtheorem{example}[theorem]{Example}

\newtheorem{remark}[theorem]{Remark}

\newtheorem{?}[theorem]{Problem}

\newcommand{\s}{\mathcal{S}}
\newcommand{\p}{\mathcal{P}}
\newcommand{\T}{\mathcal{T}}
\newcommand{\D}{\mathcal{D}}

\begin{document}

\title[Unimodal sequences and strict partitions]{On certain unimodal sequences and strict partitions}

\author[S. Fu]{Shishuo Fu}
\address[Shishuo Fu]{College of Mathematics and Statistics, Chongqing University, Huxi campus LD506, Chongqing 401331, P.R. China}
\email{fsshuo@cqu.edu.cn}

\author[D. Tang]{Dazhao Tang}

\address[Dazhao Tang]{College of Mathematics and Statistics, Chongqing University, Huxi campus LD206, Chongqing 401331, P.R. China}
\email{dazhaotang@sina.com}

\date{\today}

\begin{abstract}
Building on a bijection of Vandervelde, we enumerate certain unimodal sequences whose alternating sum equals zero. This enables us to refine the enumeration of strict partitions with respect to the number of parts and the BG-rank.
\end{abstract}



\maketitle

%\tableofcontents

%%%%%%%%%%%%%%%%%%%%%%%%%%%%%%%%%%%%%
\section{Introduction}\label{sec1: intro}
%%%%%%%%%%%%%%%%%%%%%%%%%%%%%%%%%%%%%

% come back later
A \emph{partition} \cite{Andr1976} $\lambda$ of a natural number $n$ is a finite weakly decreasing sequence of positive integers $\lambda_{1}\geq\lambda_{2}\cdots\geq\lambda_{r}>0$ such that $\sum_{i=1}^{r}\lambda_{i}=n$. The $\lambda_{i}$'s are called the \emph{parts} of the partition. Such a partition $\lambda$ is frequently represented by its \emph{Young diagram} (or \emph{Ferrers graph}) \cite[Chap. 1.3]{Andr1976}, which we take to be a left-justified array of square boxes with $r$ rows such that the $i$-th row consists of $\lambda_{i}$ boxes (the left one in Fig. \ref{shiftY}). But for the graphical representation of strict partitions, we prefer the shifted Young diagram (the right one in Fig. \ref{shiftY}).

A celebrated theorem of Euler asserts that there are as many partitions of $n$ into distinct parts as there are partitions into odd parts. In 2010, Sam Vandervelde \cite{Van} raised and then answered himself a natural question on whether there is a similar relationship between partitions into distinct parts (abbreviated as \emph{strict partitions} in what follows) and partitions involving even parts (see Corollary \ref{Vandervelde Cor} below). He provided a short generating function proof of this result by making appropriate substitutions in the Jacobi triple product identity.

\begin{figure}[htp]
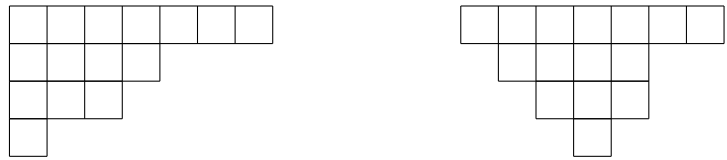
\label{shiftY}
\begin{ferrers}
\addcellrow{7}
\addcellrow{4}
\addcellrow{3}
\addcellrow{1}
\putright
\addcellrow[0]{7}
\addcellrow[1]{4}
\addcellrow[2]{3}
\addcellrow[3]{1}
\end{ferrers}
\caption{Standard and shifted Young diagrams representing $\lambda=(7,4,3,1)$}
\end{figure}

In search of a bijective proof, Vandervelde discovered this statistic which he named the ``characteristic'' of a partition, which is actually equivalent to the ``BG-rank'' introduced by Berkovich and Garvan \cite{BG} in their study of the refinements of Ramanujan's famous partition congruence modulo 5. He was then able to propose a conjecture (Conjecture 1, \cite{Van}) that not only strengthened the initial result, but also suggested strongly the existence of a bijective
proof. Vandervelde then provided such a bijective proof for the case when the characteristic or BG-rank $k=0$, for which partition he named as ``balanced partition''. The cheif aim of this paper is to prove the entire conjecture bijectively for all integer $k$. We note that, being not aware of Vandervelde's conjecture, Huang et al. \cite{HSWW} supplied essentially the first proof of it (see Corollary \ref{HSWW cor} below).

The rest of the paper is organized as follows. We introduce and study certain unimodal sequences in Section \ref{sec2: main theorem}. In Section \ref{sec: application}, we build on this result and enumerate strict partitions with respect to the number of parts and the BG-rank (see Theorem \ref{4case}). We conclude in the last section with some remarks to motivate further investigation.

%%%%%%%%%%%%%%%
\section{$(a,b)$-sequences}\label{sec2: main theorem}
%%%%%%%%%%%%%%%
\begin{Def}\label{abseq}
For integers $a\ge 0,\; l\ge b\ge 1$, we call $\{d_1,\cdots,d_l\}$ an \emph{$(a,b)$-sequence of length $l$}, if it consists of $l$ positive integers that satisfy the following conditions:
\begin{enumerate}[1)]
\item $d_1=a+1, d_2=a+2,\cdots,d_b=a+b~$;
\item $d_b\ge d_{b+1}\ge d_{b+2}\ge \cdots \ge d_l\ge 1~$;
\item $\sum_{i=1}^{l}(-1)^{i}d_i=0~$.
\end{enumerate}
\end{Def}

Let $\s_{a,b}$ denote the set of all $(a,b)$-sequences, and denote $\s:=\left(\bigcup_{a\ge 0, b\ge 1}\s_{a,b}\right)\cup \{\epsilon\}$, where $\epsilon$ denotes the empty sequence. Suppose $\Delta=\{d_1,\cdots,d_l\}\in \s_{a,b}~$, then let $a(\Delta)=a,\; b(\Delta)=b,\; l(\Delta)=l,\; |\Delta|=\sum_{i=1}^{l}d_i,\; |\Delta|_a=\sum_{i=1}^{l}(-1)^id_i$. For the empty sequence $\epsilon$, we set $a(\epsilon)=b(\epsilon)=l(\epsilon)=|\epsilon|=|\epsilon|_a=0$. Notice that due to condition 3), $|\Delta|$ must be an even integer. We have the following bivariate generating function.

\begin{theorem}\label{gf: abseq}
There holds
\begin{align}\label{id: abseq}
S(x,y) & :=\sum_{\Delta\in\s}x^{a(\Delta)}y^{b(\Delta)}q^{\frac{|\Delta|}{2}}=1+\sum_{i=0}^{\infty}x^i\sum_{k=1}^{\infty}\frac{(1+(y-1)q^k)y^{2k-1}q^{k(i+k)}}
{(q;q)_k(q;q)_{i+k}},
\end{align}
where
\begin{align*}
(a;q)_{k}=\prod_{n=0}^{k-1}(1-aq^{n})
\end{align*}
is the $q$-shifted factorial.
\end{theorem}

In order to prove the above theorem, we constuct a bijection $\phi_a$ that maps each $\Delta\in \s_{a,b}$ with weight $|\Delta|=2n$ to an integer partition $\lambda$ of $n$. This bijection, in the case of $a=0$, was originally developed by Vandervelde \cite{Van}. Before we describe this map, let us make some useful observations.

\begin{lemma}\label{fillup}
Take any $\Delta=\{d_1,\cdots,d_l\}\in\s_{a,b}$.
\begin{enumerate}
\item There does not exist $m,\; 0<m<l$, such that
\begin{align}\label{same sign}
\left(\sum_{i=1}^{m}(-1)^i d_i\right)\cdot\left(\sum_{i=1}^{m+1}(-1)^i d_i\right)>0.
\end{align}
\item If for some $n\ge b-1$, we have $\sum_{i=1}^{n}(-1)^i d_i=0$, then $n\equiv l \pmod 2$, and
\begin{align}\label{pair}
d_{n+1}=d_{n+2}\ge d_{n+3}=d_{n+4}\ge \cdots \ge d_{l-1}=d_l.
\end{align}
In particular, when $b=1$, we can take $n=0$, and $\sum_{i=1}^{0}(-1)^i d_i=0$ holds trivially, so in this case we always have $d_1=d_2\ge d_3=d_4\ge \cdots \ge d_{l-1}=d_l$.
\end{enumerate}
\end{lemma}

\begin{proof}
First, we prove (1) by contradiction. Suppose there exists such a $\Delta\in \s_{a,b}$, find the smallest $m$ that satisfies \eqref{same sign}. A simple calculation reveals that $m\ge b$. Next we discuss according to the parity of $m$.
\begin{itemize}
\item[I.] If $m$ is even, then by the minimality of $m$, we must have $$\sum_{i=1}^{m}(-1)^i d_i>0,\; \sum_{i=1}^{m+1}(-1)^i d_i>0,$$  But now we see
\begin{align*}
0=\sum_{i=1}^l (-1)^i d_i=\sum_{i=1}^{m+1} (-1)^i d_i+(d_{m+2}-d_{m+3}+\cdots)\ge \sum_{i=1}^{m+1} (-1)^i d_i>0,
\end{align*}
a contradiction. The penultimate inequality is due to condition 2) in Definition~\ref{abseq}.
\item[II.] If $m$ is odd, then we have instead $$\sum_{i=1}^{m}(-1)^i d_i<0,\; \sum_{i=1}^{m+1}(-1)^i d_i<0,$$ and we see
\begin{align*}
0=\sum_{i=1}^l (-1)^i d_i=\sum_{i=1}^{m+1} (-1)^i d_i-(d_{m+2}-d_{m+3}+\cdots)\le \sum_{i=1}^{m+1} (-1)^i d_i<0,
\end{align*}
which is absurd again.

Next for (2), since $n\ge b-1$, $n+1\ge b$, so we have $d_{n+1}\ge d_{n+2}\ge \cdots \ge d_l\ge 1$. Moreover, condition 3) together with $\sum_{i=1}^{n}(-1)^i d_i=0$ forces $\sum_{i=n+1}^l (-1)^i d_i=0$, which in turn results in \eqref{pair}, as desired.
\end{itemize}
\end{proof}

Following Chu \cite{Chu}, we define a $k$-Durfee rectangle for the Young diagram of a given partition to be the largest $i\times (i+k)$ rectangle (consisting of $i$ rows and $i+k$ columns) contained in that Young diagram for a fixed $k$. Then the original Durfee square becomes the $0$-Durfee rectangle in this setting. Notice that this notion of Durfee rectangle is different from Andrews' generalization in \cite{And}.

\begin{Def}\label{pab}
Let $\p$ be the set of all (integer) partitions, including the empty partition $\varnothing$. And for $a\ge 0, b\ge 1$, let $\p_{a,b}$ be the set of all integer partitions $\lambda=(\lambda_1,\lambda_2,\cdots)$ whose $a$-Durfee rectangle has size $\left\lceil \frac{b}{2}\right\rceil \times (\left\lceil \frac{b}{2}\right\rceil+a)$, and either $\lambda_{b/2}>a+b/2$ if $b$ is even, or $\lambda_{(b+1)/2}=a+(b+1)/2$ if $b$ is odd.
\end{Def}

Now we describe the aforementioned bijection $\phi_a$. The main idea is to use the given sequence $\Delta=\{d_1,\cdots,d_l\}\in \s_{a,b}$ to ``double cover'' the board configuration depicted in Fig.~\ref{board}. The doubly covered region in the end of this process will be the Young diagram of an integer partition in $\p_{a,b}$.

More precisely, the $i$-th labelled block in this board configuration has size $1\times \left(a+\dfrac{i+1}{2}\right)$ (resp. $\dfrac{i}{2}\times 1$) if $i$ is odd (resp. even). We denote this block as $B_i$ and its area as $b_i$. So for example, $b_1=a+1,b_2=1,b_3=a+2,b_4=2,\cdots$. Assume we are given $\Delta=\{d_1,\cdots,d_l\}\in \s_{a,b}$, we get $\phi_a(\Delta)$ by performing the following fillings of the initially empty board.
\begin{enumerate}
\item Use $d_1=a+1$ cells to fill up $B_1$.
\item For $2\le i\le l$, first use $d_{i-1}$ cells to double cover the already existing cells in $B_{i-1}$, and then use the remaining cells to fill $B_i$.
\item For each odd-numbered (horizontal) block, the filling is always from left to right, while for each even-numbered (vertical) block, the filling is from top to bottom.
\item After we have used up all $d_i, 1\le i\le l$, the doubly covered cells form the Young diagram of an integer partition, which is defined to be $\lambda=\phi_a(\Delta)$.
\end{enumerate}

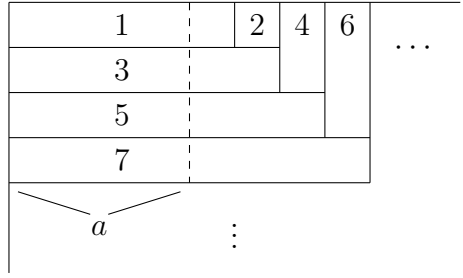
\begin{figure}
\begin{tikzpicture}[scale=0.6]
\draw (0,0) --  (10,0);
\draw (0,-1) --  (6,-1);
\draw (0,-2) --  (7,-2);
\draw (0,-3) --  (8,-3);
\draw (0,-4) --  (8,-4);
\draw (0,0) --  (0,-6);
\draw[dashed] (4,0) --  (4,-4);
\draw (5,0) --  (5,-1);
\draw (6,0) --  (6,-2);
\draw (7,0) --  (7,-3);
\draw (8,0) --  (8,-4);
\draw (0.2,-4.2) --  (1.8,-4.7);
\draw (3.8,-4.2) --  (2.2,-4.7);
% label
\node (1) at (2.5,-0.5) {$1$};
\node (2) at (5.5,-.5) {$2$};
\node (3) at (2.5,-1.5) {$3$};
\node (4) at (6.5,-.5) {$4$};
\node (5) at (2.5,-2.5) {$5$};
\node (6) at (7.5,-.5) {$6$};
\node (7) at (2.5,-3.5) {$7$};
\node (a) at (2,-5) {$a$};
\node at (5,-5) {$\vdots$};
\node at (9,-1) {$\cdots$};
\end{tikzpicture}
\caption{Board configuration for $\phi_a$ with labelled blocks}
\label{board}
\end{figure}

\begin{theorem}\label{bij}
For a fixed $a\ge 0$ and any $b\ge 1$, the map $\phi_a$ defined above is a bijection from $\s_{a,b}$ to $\p_{a,b}$, such that $|\Delta|=2|\phi_a(\Delta)|$, for any $\Delta\in \s_{a,b}$.
\end{theorem}
\begin{proof}
Given any $\Delta\in \s_{a,b}$, to see that $\phi_a(\Delta)$ is indeed in $\p_{a,b}$, we consider two phases.

Phase I is when we use $d_i$ cells to fill the blocks for $1\le i\le b$. Each time we observe an exact cover of the blocks involved, since $b_1=a+1=d_1, b_1+b_2=a+2=d_2, b_2+b_3=a+3=d_3,\cdots, b_{b-1}+b_b=a+b=d_b$. In the end of this phase we have built up the $\left\lceil \frac{b}{2}\right\rceil \times (\left\lceil \frac{b}{2}\right\rceil+a)$ Durfee rectangle of $\phi_a(\Delta)$. Moreover, depending on the parity of $b$, we have certain restriction on the length of the bottom row in this Durfee rectangle, as stated in Definition~\ref{pab}.

Phase II is for $b<i\le l$. We claim that at the $i$-th step, $d_i$ is no less than the number of cells already covered in $B_{i-1}$, and no more than the number of cells already doubly covered in $B_{i-2}$. More precisely, we have
\begin{itemize}
\item $d_i\ge |\sum_{j=1}^{i-1}(-1)^j d_j|$. This is due to Lemma~\ref{fillup} $(1)$.
\item $d_{i-1}\ge d_i$. This is clear since $i>b$ and we have condition 2) in Definition~\ref{abseq}.
\end{itemize}
Finally, condition 3) in Definition~\ref{abseq} guarantees that we will not have any singly covered cells left. In conclusion, $\lambda=\phi_a(\Delta)\in \p_{a,b}$ indeed, and each cell of its Ferrers graph is doubly covered, when we used up all $d_i$s in $\Delta$. This shows that the map $\phi_a$ is well-defined, and $|\Delta|=2|\phi_a(\Delta)|$.

To prove that $\phi_a$ is a bijection, we note the following way to define its inverse. Given an integer partition $\lambda\in\p_{a,b}$, we put its Ferrers graph on the board as shown in Fig.~\ref{board}, such that their top left corners coincide. And we denote the number of cells in $\lambda$ that are confined in $B_i$ as $c_i$, then we put
$$d_1=c_1,~d_i=c_{i-1}+c_{i}, \text{ for } i\ge 2.$$
It should be clear that the positive numbers $d_1,d_2,\cdots$ form a sequence in $\s_{a,b}$, and this is exactly $\phi_a^{-1}(\lambda)$. This completes the proof.
\end{proof}

\begin{example}
Assume $a=6$, $b=5$, take $\Delta=(7, 8, 9, 10, 11, 11, 8, 7, 5 ,5, 4, 3, 1, 1)\in\s_{6,5}$, then $\lambda=\phi_6(\Delta)=(12, 10, 9, 6, 4, 3, 1)\in\mathcal{P}_{6,5}$ and $|\lambda|=|\Delta|/2=45$.
\end{example}

\begin{corollary}
Theorem~\ref{gf: abseq} is true.
\end{corollary}
\begin{proof}
First note that the empty sequence corresponds to the empty partition trivially, and they are both weighted as $1$ from both sides of the identity. Next for any $b\ge 1$, fix an $a\ge 0$, it suffices to compare the coefficients of $x^a$ on both sides of \eqref{id: abseq}, and realize that they are respectively the generating functions of $(a,b)$-sequences over $\bigcup_{b\ge 1}\s_{a,b}$, and integer partitions over $\bigcup_{b\ge 1}\p_{a,b}$, which are the same according to Theorem~\ref{bij}. To obtain on the right hand side of the identity, we only need to add together the odd and even cases for the power of $y$ and simple simplifications.
\end{proof}

%%%%%%%%%%%%%%%%%%
\section{Application to strict partitions}\label{sec: application}
%%%%%%%%%%%%%%%%%%
The motivation for studying $(a,b)$-sequences introduced in the last section comes from their connection with strict partitions. Let $\D$ denote the set of all strict partitions, including the empty partition $\varnothing$ and $\T=\{\frac{1}{2}k(k+1): k\in\mathbb{Z}\}$ denote the set of all triangular numbers, then the following observation reveals this connection.

\begin{lemma}\label{decomp}
There is an injection $\iota:\; \D\rightarrow \T\times \s$. Suppose $\iota(\lambda)=(t,\Delta)$, then $|\lambda|=t+|\Delta|$. Moreover, $(t,\Delta)\in \iota(\D)$ if and only if either
\begin{enumerate}
\item $a(\Delta)=k$, or
\item $a(\Delta)\le k-1$ and $b(\Delta)=1$, or
\item $\Delta=\epsilon$,
\end{enumerate}
where $t=\binom{k+1}{2}$ for some $k\ge 0$.
\end{lemma}
\begin{proof}
Recall that we use the shifted Young diagram to represent strict partitions as the second diagram in Fig.~\ref{shiftY}. Given a partition $\lambda=(\lambda_1,\lambda_2,\cdots,\lambda_m)\in \D$, we read its shifted Young diagram column-wise from left to right and obtain a sequence $c(\lambda)=\{c_1,c_2,\cdots,c_{\lambda_1}\}$. For instance, we get $c(\lambda)=\{1,2,3,4,3,1,1\}$ for the diagram shown in Fig.~\ref{shiftY}. In general, it should be clear that $c(\lambda)$ satisfies the first two conditions of being a $(0,m)$-sequence. Moreover, we claim that there exists certain integer $k$, where $0\le k \le m$, such that $\sum_{i=1}^{k}(-1)^{i}c_i=\sum_{i=1}^{\lambda_1}(-1)^{i}c_i$. Assuming this claim for the moment, we see that $\Delta:=\{c_{k+1},c_{k+2},\cdots,c_{\lambda_1}\}$ satisfies all three conditions in Definition~\ref{abseq} and hence $\Delta\in\s$. Now we simply take $t=1+2+\cdots+k=\binom{k+1}{2}$, and define $\iota(\lambda)=(t,\Delta)$. We see $\iota(\lambda)\in \T\times \s$ indeed and
$$|\lambda|=\lambda_1+\lambda_2+\cdots+\lambda_m=c_1+c_2+\cdots+c_{\lambda_1}=t+|\Delta|.$$
Now take any $(t,\Delta)\in\iota(\D)$ with $t=\binom{k+1}{2}$, we have the following discussion, which leads to the three cases characterizing the image set $\iota(\D)$.
\begin{enumerate}
\item $k<m$. In this case $k+1\le m$ so $c_{k+1}=k+1$, which means $a(\Delta)=k$.
\item $k=m<\lambda_1$. In this case $k=c_k\ge c_{k+1}\ge c_{k+2}\geq\cdots$, so we must have $a(\Delta)\le k-1$ and $b(\Delta)=1$.
\item $k=m=\lambda_1$. This happens exactly when $\lambda$ itself is a staircase partition $$(k,k-1,\cdots,2,1)$$ and $\Delta=\epsilon$.
\end{enumerate}

To see that $\iota$ is an injection, simply note that for any $(t,\Delta)\in \D\times\s$ that satisfies either (1) or (2) or (3), we can uniquely recover its preimage by appending columns of length $1,2,\cdots,k$ to the left of the columns of length given by the integers in the sequence $\Delta$ and getting a valid shifted Young diagram. This also proves the ``if'' part of the characterization of $\iota(\D)$.

Finally we prove the aforementioned claim. For $c(\lambda)=\{c_1,c_2,\cdots,c_{\lambda_1}\}$, we consider the sequence of the initial alternating sums, namely
\begin{align}
0, \quad -c_1, \quad -c_1+c_2, \quad -c_1+c_2-c_3, \quad -c_1+c_2-c_3+c_4, \quad \ldots
\end{align}
This sequence ends at the full alternating sum $|c(\lambda)|_a$, and it starts as $0,-1,1,-2,2,\ldots$, until it reaches $\sum_{i=1}^{m}(-1)^ic_i$. This is also the point where this alternating sum assumes the largest absolute value (say $l$) among the entire sequence. Actually, by far all the integer values from $-l$ to $l$ have appeared precisely once each, with a possible exception for value $l$ (this depends on the parity of $m$). Passing this point, the behavior of these alternating sums is less predictable but will never surpass $l$ in absolute value, since $c_m\ge c_{m+1}\ge\cdots\ge c_{\lambda_1}$. Therefore, the ending value $|c(\lambda)|_a$ must have appeared exactly once before $\sum_{i=1}^{m}(-1)^{i}c_{i}$, hence we have proved the claim.
\end{proof}
\begin{example}
Tab. \ref{D:table} lists all five strict partitions of $7$ and their images under the injection $\iota$.
\begin{table}[tbp]\caption{the injection $\iota$}\label{D:table}
\centering
\begin{tabular}{c|c}
\hline
 $\D$ & $\iota(\D)$\\
\hline
$7$ & $(1,\{1,1,1,1,1,1\})$ \\
$6+1$ & $(3,\{1,1,1,1\})$ \\
$5+2$ & $(1,\{2,2,1,1\})$ \\
$4+3$ & $(3,\{2,2\})$ \\
$4+2+1$ & $(1,\{2,3,1\})$ \\
\hline
\end{tabular}
\end{table}
\end{example}
Basing on this injection, we can enumerate strict partitions with respect to the number of parts and the BG-rank.
\begin{Def}[Berkovich-Garvan \cite{BG}, 2006]
For any partition $\lambda$, its BG-rank is defined to be the excess of the number of the odd-indexed odd parts over the number of the even-indexed odd parts, and is denoted as $r(\lambda)$.
\end{Def}
\begin{remark}\label{twodef}
If $\lambda$ is a strict partition, we have $-|c(\lambda)|_a=r(\lambda)$. One convenient way to see this is to paint the (standard) Young diagram in a chessboard fashion using black and white, and then weight each black cell (resp. white cell) by $1$ (resp. by $-1$), with the left-top cell painted black. Then the sum of the weights of all the cells in the diagram gives us $r(\lambda)$ on one hand, and when viewed using the shifted Young diagram, it gives us $-|c(\lambda)|_a$ on the other hand.
\end{remark}
\begin{theorem}\label{4case}
Given any integers $k,m,n$ such that $m\ge 2k-1$ if $k>0$, $m\ge -2k$ if $k\le 0$, $n\ge \binom{m+1}{2}$, the number of strict partitions of $n$ with exactly $m$ parts and BG-rank being $k$, equals
\begin{enumerate}
\item the number of $\Delta \in \s_{2k-1,m-2k+1}$ with $|\Delta|=n-\binom{2k}{2}$, if $k>0, m>2k-1$.
\item the number of partitions $\lambda$ with its largest part no greater than $2k-1$, and $2|\lambda|=n-\binom{2k}{2}$, if $k>0, m=2k-1$.
\item the number of $\Delta \in \s_{-2k,m+2k}$ with $|\Delta|=n-\binom{2k}{2}$, if $k\le 0, m>-2k$.
\item the number of partitions $\lambda$ with its largest part no greater than $-2k$, and $2|\lambda|=n-\binom{2k}{2}$, if $k\le 0, m=-2k$.
\end{enumerate}
\end{theorem}
\begin{proof}
We only show case (1) here, since the remaining cases are quite similar. As we have already observed in Remark~\ref{twodef}, for any strict partition $\lambda$, we have $r(\lambda)=-|c(\lambda)|_a$. Now suppose $k>0, m>2k-1, n\ge \binom{m+1}{2}$, and $\lambda$ is a strict partition of $n$ with $m$ parts and $r(\lambda)=k$, so $|c(\lambda)|_a=-k<0$. We apply $\iota$ to get $\iota(\lambda)=(t,\Delta)$, where $t=\binom{2k}{2}$ and $\Delta\in\s_{2k-1,m-2k+1}$. Conversely, for any $\Delta\in\s_{2k-1,m-2k+1}$ with $|\Delta|=n-\binom{2k}{2}$, we can append $1,2,\cdots,2k-1$ to the left of $\Delta$ to get $c(\lambda)$, and hence $\lambda$, so these two sets are indeed equinumerous. Finally, note that in cases (2) and (4), we use partitions rather than $(a,b)$-sequences, because in these two cases we have $b(\Delta)=1$ and the observation in Lemma~\ref{fillup}~(2) gives us the results.
\end{proof}
\begin{example}
We enumerate here all the strict partitions of $n$ with $m$ parts and BG-rank being $k$, for various values of $(k,m,n)$. And we also list the corresponding equinumerous objects as predicted by the last theorem.
\begin{enumerate}
\item For $(k,m,n)=(3,6,33)$, we have the following qualified strict partitions of $33$:
$$(13,6,5,4,3,2),\; (11,8,5,4,3,2),\; (9,8,7,4,3,2).$$
And there are exactly three sequences in $\s_{5,1}$ with weight being $n-\binom{2k}{2}=18$, namely
$$\{6,6,1,1,1,1,1,1\},\; \{6,6,2,2,1,1\},\; \{6,6,3,3\}.$$
\item For $(k,m,n)=(2,3,16)$, we have the following qualified strict partitions of $16$:
$$(13,2,1),\; (11,4,1),\; (9,6,1),\; (9,4,3),\; (7,6,3).$$
And there are exactly five partitions of $\dfrac{n-\binom{2k}{2}}{2}=5$ with the largest part no greater than $3$, namely
$$(1,1,1,1,1),\; (2,1,1,1),\; (2,2,1),\; (3,1,1),\; (3,2).$$
\item For $(k,m,n)=(0,3,12)$, we have the following qualified strict partitions of $12$:
$$(6,4,2),\; (7,3,2),\; (8,3,1),\; (6,5,1).$$
And there are exactly four sequences in $\s_{0,3}$ with weight being $n-\binom{2k}{2}=12$, namely
$$\{1,2,3,3,2,1\},\; \{1,2,3,3,1,1,1\},\; \{1,2,3,2,1,1,1,1\},\; \{1,2,3,2,2,2\}.$$
\item For $(k,m,n)=(-1,2,11)$, we have the following qualified strict partitions of $1$:
$$(10,1),\; (8,3),\; (6,5).$$
And there are exactly three partitions of $\dfrac{n-\binom{2k}{2}}{2}=4$ with the largest part no greater than $2$, namely
$$(1,1,1,1),\; (2,1,1),\; (2,2).$$
\end{enumerate}
\end{example}
As an immediate application of last theorem, we derive the following result first obtained by Huang et al in \cite{HSWW}.
\begin{corollary}[Proposition 5.13, \cite{HSWW}]\label{HSWW cor}
Let $\p_j$ be the set of all partitions with $2$-core size $\binom{2j}{2}$, then we have
\begin{align}\label{pjd}
\sum_{\lambda\in\p_j\cap\D}q^{|\lambda|}=\frac{q^{\binom{2j}{2}}}{\prod_{i=1}^{\infty}(1-q^{2i})}.
\end{align}
\end{corollary}
\begin{proof}
As observed in \cite{CJW}, $r(\lambda)=j$ if and only if its $2$-core is a staircase partition of height $2j-1$, if $j>0$, and $-2j$, if $j\le 0$. In either case, the size of its $2$-core is $\binom{2j}{2}$. Therefore $\p_j\cap \D$ is exactly the set of all strict partitions with BG-rank being $j$.

To get its generating function, we just fix $j$ (suppose $j>0$), and let $m$ run through all integers no less than $2j-1$, i.e., we combine cases (1) and (2) in Theorem~\ref{4case} for all possible values of $m$. We note that $\p_{a,2b}\cup\p_{a,2b-1}$ is the set of partitions with their $a$-Durfee rectangles having size $b\times(b+a)$, and the partitions without an $a$-Durfee rectangle are those whose largest part is no greater than $a$. This, together with Theorem~\ref{bij}, gives us \eqref{pjd}. The $j\le 0$ case can be proved analogously, we just need to combine cases (3) and (4) in Theorem~\ref{4case} instead.
\end{proof}

Now we just sum over all $j$ and compare the coefficients of $q^{n}$ on both sides to arrive at Vandervelde's partition theorem that parallels Euler's ``Distinct v.s. Odd'' theorem.
\begin{corollary}[Proposition 1, \cite{Van}]\label{Vandervelde Cor}
For every nonnegative integer $n$, the number of strict partitions of $n$ is equal to the number of partitions of $n$ into even parts along with precisely one triangular part.
\end{corollary}

\section{Final remarks}
We conclude with several remarks that merit further study.
\begin{enumerate}[1)]
\item First, with the aid of \eqref{pjd}, we see that
\begin{align}\label{iden:pjd}
q_{j}(n)=p\left(\dfrac{n-j(2j-1)}{2}\right),
\end{align}
where $q_{j}(n)$ counts the number of strict partitions of $n$ with BG-rank being $j$. Combining \eqref{iden:pjd} and one of Ramanujan's classical congruences \cite{Ram} $p(5n+4)\equiv0\pmod{5}$, we obtain the following:
\begin{align}
q_{j}(10n+1) &\equiv0\pmod{5},\quad \textrm{if}\quad j\equiv 9\phantom{,0} \pmod{10},\label{cong 1:mod 10}\\
q_{j}(10n+3)&\equiv0\pmod{5},\quad \textrm{if}\quad j\equiv 3,5\pmod{10},\label{cong 2:mod 10}\\
q_{j}(10n+4)&\equiv0\pmod{5},\quad \textrm{if}\quad j\equiv 2,6\pmod{10},\label{cong 3:mod 10}\\
q_{j}(10n+6)&\equiv0\pmod{5},\quad \textrm{if}\quad j\equiv 4\phantom{,0}\pmod{10},\label{cong 4:mod 10}\\
q_{j}(10n+8)&\equiv0\pmod{5},\quad \textrm{if}\quad j\equiv 0,8\pmod{10},\label{cong 5:mod 10}\\
q_{j}(10n+9)&\equiv0\pmod{5},\quad \textrm{if}\quad j\equiv 1,7\pmod{10}.\label{cong 6:mod 10}
\end{align}

Of course, there are more congruences beyond this list. For example, by Ramanujan's congruences modulo 7 and 11 for $p(n)$, we will obtain some congruences modulo 7 and 11 for $q_{j}(n)$ similar to \eqref{cong 1:mod 10}--\eqref{cong 6:mod 10}.

\item Next, there are many refinements of Euler's ``Distinct v.s. Odd'' theorem. As an example we present the following refinement due to Fine.
\begin{theorem}[Theorem 2.13, \cite{Andr1976}]
The number of strict partitions of $n$ with largest part $k$ equals the number of partitions of $n$ into odd parts such that $2k+1$ equals the largest part plus twice the number of parts.
\end{theorem}

The BG-rank naturally induces a new refinement of the set of strict partitions, it would be interesting to find an equidistributed statistic on the set of partitions into odd parts. Consequently this will lead to a new refinement of Euler's theorem.

\item Finally, recall that the classical hook length formula for partitions, ordinary and strict.
\begin{align}
\sum_{\lambda\in\mathcal{P}}x^{|\lambda|}\prod_{h\in\mathcal{H}(\lambda)}\dfrac{1}{h^{2}}=e^{x}, \label{hook len for: ordinary}\\
\sum_{\mu\in\mathcal{D}}x^{|\mu|}\prod_{h\in\mathcal{H}(\mu)}\dfrac{1}{2^{\ell(\mu)}h^{2}}=e^{\frac{x}{2}}, \label{hook len for: strict}
\end{align}
where $\mathcal{H}(\lambda)$ (resp. $\mathcal{H}(\mu)$) is the multiset of hook lengths associated with the ordinary partition $\lambda$ (resp. the strict partition $\mu$), and $\ell(\mu)$ counts the number of parts in $\mu$. One of the generalization of \eqref{hook len for: ordinary} given in \cite[Theorem 4.2]{HJ2011} is
\begin{align}
\sum_{\lambda\in\mathcal{P}}q^{|\lambda|}x^{\#\mathcal{H}_{2}(\lambda)}b^{BG(\lambda)}\prod_{h\in\mathcal{H}_{2}(\lambda)}\dfrac{1}{h^{2}}=\exp
\left(\dfrac{xq^{2}}{2}\right)\sum_{j=-\infty}^{+\infty}b^{j}q^{j(2j-1)},\label{BG-rank version}
\end{align}
where $\mathcal{H}_{2}(\lambda)=\{h\in\mathcal{H}(\lambda),~h\equiv0\pmod{2}\}$. In view of this, it is natural to consider a generalization of \eqref{hook len for: strict} to involve BG-rank as \eqref{BG-rank version} does. Although at this moment we are unclear how this could be done, since the hook length of strict partition is essentially defined as the usual hook length of the corresponding doubled distinct partition~\cite{HX2016}, while the BG-rank of all doubled distinct partitions are zero.
\end{enumerate}

\section*{Acknowledgement}
The authors wish to acknowledge Dennis Stanton for bringing~\cite{HSWW} to their attention, and acknowledge Guo-Niu Han and Huan Xiong for their helpful comments on a preliminary version of this paper. Both authors' research were supported by the National Science Foundation of China grant 11501061.

\end{document}